\documentclass[11pt]{amsart}
\usepackage{amsxtra}
\usepackage{graphics}
\usepackage{latexsym}
\usepackage{enumerate}
\usepackage{xcolor}
\usepackage{amsmath}
\usepackage{amssymb,amsthm,amsfonts}
\usepackage{mathtools}
\usepackage{amscd}
\usepackage[arrow, matrix, curve]{xy}
\usepackage[mathscr]{euscript}
\usepackage{syntonly}
\usepackage{ifpdf}
\usepackage{hyperref}
\title[ ]{An explicit infinite homotopy in nonabelian Hodge theory in positive characteristic}

\author[Mao Sheng]{Mao Sheng}
\email{msheng@ustc.edu.cn}
\address{School of Mathematical Sciences,
University of Science and Technology of China, Hefei, 230026, China}
\email{msheng@tsinghua.edu.cn}
	\address{Yau Mathematical Science Center, Tsinghua University, Beijing, 100084, China.}
\author[Zebao Zhang]{Zebao Zhang}
\email{zhangzebao@bicmr.pku.edu.cn}
\address{
Beijing International Center for Mathematical Research, 
Peking University,
Beijing, 100871, China}

\begin{document}
\theoremstyle{plain}
\newcommand{\rev}[1]{\color{red}#1\color{black}}\
\newtheorem{thm}{Theorem}[section]
\newtheorem{theorem}[thm]{Theorem}
\newtheorem*{theorem*}{Theorem}
\newtheorem*{theoremA*}{Theorem A}
\newtheorem*{theoremB*}{Theorem B}
\newtheorem*{theoremC*}{Theorem C}
\newtheorem*{definition*}{Definition}
\newtheorem{lemma}[thm]{Lemma}
\newtheorem{sublemma}[thm]{Sublemma}
\newtheorem{corollary}[thm]{Corollary}
\newtheorem*{corollary*}{Corollary}
\newtheorem{proposition}[thm]{Proposition}
\newtheorem{addendum}[thm]{Addendum}
\newtheorem{variant}[thm]{Variant}
\theoremstyle{definition}
\newtheorem{lemma-definition}[thm]{Lemma-Definition}
\newtheorem{proposition-definition}[thm]{Proposition-Definition}
\newtheorem{theorem-definition}[thm]{Theorem-Definition}
\newtheorem{construction}[thm]{Construction}
\newtheorem{notations}[thm]{Notations}
\newtheorem{question}[thm]{Question}
\newtheorem{problem}[thm]{Problem}
\newtheorem*{problem*}{Problem}
\newtheorem{remark}[thm]{Remark}
\newtheorem*{remark*}{Remark}
\newtheorem{remarks}[thm]{Remarks}
\newtheorem{definition}[thm]{Definition}
\newtheorem{claim}[thm]{Claim}
\newtheorem{assumption}[thm]{Assumption}
\newtheorem{assumptions}[thm]{Assumptions}
\newtheorem{properties}[thm]{Properties}
\newtheorem{example}[thm]{Example}
\newtheorem{conjecture}[thm]{Conjecture}
\numberwithin{equation}{thm}

\newcommand{\et}{\mathrm{\acute{e}t}}
\newcommand{\zar}{\mathrm{zar}}
\newcommand{\sA}{{\mathcal A}}
\newcommand{\sB}{{\mathcal B}}
\newcommand{\sC}{{\mathcal C}}
\newcommand{\sD}{{\mathcal D}}
\newcommand{\sE}{{\mathcal E}}
\newcommand{\sF}{{\mathcal F}}
\newcommand{\sG}{{\mathcal G}}
\newcommand{\sH}{{\mathcal H}}
\newcommand{\sI}{{\mathcal I}}
\newcommand{\sJ}{{\mathcal J}}
\newcommand{\sK}{{\mathcal K}}
\newcommand{\sL}{{\mathcal L}}
\newcommand{\sM}{{\mathcal M}}
\newcommand{\sN}{{\mathcal N}}
\newcommand{\sO}{{\mathcal O}}
\newcommand{\sP}{{\mathcal P}}
\newcommand{\sQ}{{\mathcal Q}}
\newcommand{\sR}{{\mathcal R}}
\newcommand{\sS}{{\mathcal S}}
\newcommand{\sT}{{\mathcal T}}
\newcommand{\sU}{{\mathcal U}}
\newcommand{\sV}{{\mathcal V}}
\newcommand{\sW}{{\mathcal W}}
\newcommand{\sX}{{\mathcal X}}
\newcommand{\sY}{{\mathcal Y}}
\newcommand{\sZ}{{\mathcal Z}}
\newcommand{\A}{{\mathbb A}}
\newcommand{\B}{{\mathbb B}}
\newcommand{\C}{{\mathbb C}}
\newcommand{\D}{{\mathbb D}}
\newcommand{\E}{{\mathbb E}}
\newcommand{\F}{{\mathbb F}}
\newcommand{\G}{{\mathbb G}}
\newcommand{\HH}{{\mathbb H}}
\newcommand{\I}{{\mathbb I}}
\newcommand{\J}{{\mathbb J}}
\renewcommand{\L}{{\mathbb L}}
\newcommand{\M}{{\mathbb M}}
\newcommand{\N}{{\mathbb N}}
\renewcommand{\P}{{\mathbb P}}
\newcommand{\Q}{{\mathbb Q}}
\newcommand{\R}{{\mathbb R}}
\newcommand{\SSS}{{\mathbb S}}
\newcommand{\T}{{\mathbb T}}
\newcommand{\U}{{\mathbb U}}
\newcommand{\V}{{\mathbb V}}
\newcommand{\W}{{\mathbb W}}
\newcommand{\X}{{\mathbb X}}
\newcommand{\Y}{{\mathbb Y}}
\newcommand{\Z}{{\mathbb Z}}
\newcommand{\id}{{\rm id}}
\newcommand{\rank}{{\rm rank}}
\newcommand{\END}{{\mathbb E}{\rm nd}}
\newcommand{\End}{{\rm End}}
\newcommand{\Hom}{{\rm Hom}}
\newcommand{\Hg}{{\rm Hg}}
\newcommand{\tr}{{\rm tr}}
\newcommand{\Sl}{{\rm Sl}}
\newcommand{\Gl}{{\rm Gl}}
\newcommand{\Cor}{{\rm Cor}}
\newcommand{\Aut}{\mathrm{Aut}}
\newcommand{\Sym}{\mathrm{Sym}}
\newcommand{\ModuliCY}{\mathfrak{M}_{CY}}
\newcommand{\HyperCY}{\mathfrak{H}_{CY}}
\newcommand{\ModuliAR}{\mathfrak{M}_{AR}}
\newcommand{\Modulione}{\mathfrak{M}_{1,n+3}}
\newcommand{\Modulin}{\mathfrak{M}_{n,n+3}}
\newcommand{\Gal}{\mathrm{Gal}}
\newcommand{\Spec}{\mathrm{Spec}}
\newcommand{\res}{\mathrm{Res}}
\newcommand{\coker}{\mathrm{coker}}
\newcommand{\Jac}{\mathrm{Jac}}
\newcommand{\HIG}{\mathrm{HIG}}
\newcommand{\MIC}{\mathrm{MIC}}

\thanks{The work is supported by National Key Research and Development Project SQ2020YFA070080, CAS Project for Young Scientists in Basic Research Grant No. YSBR-032, National Natural Science Foundation of China (Grant No. 11721101), Fundamental Research Funds for the Central Universities.}
\begin{abstract}
This short note is extracted from \cite[\S3, Appendix]{SZ}, where an explicit infinite homotopy from a Higgs complex to the Frobenius pushforward of the corresponding de Rham complex in positive characteristic has been provided. The verification details, which are omitted therein, are provided here.  
\end{abstract} 

\maketitle
 
In \cite{SZ}, the notion of an $\sL$-indexed $\infty$-homotopy between complexes was introduced. Let us recall it briefly here for the readers' convenience. Let $(Y,\mathcal O_Y)$ be a ringed space. For $\mathcal F^*,\mathcal G^*$, two complexes of sheaves of $\mathcal O_Y$-modules, 
$$
(\mathcal Hom^*_{\mathcal O_Y}(\mathcal F^*,\mathcal G^*),d_{\mathcal Hom})
$$ denotes for the associated Hom complex. Let $\sL$ be a sheaf of sets over $Y$ whose stalks are all nonempty. Let $\Delta_*(\mathcal L)$ be the simplicial complex attached to $\mathcal L$: For $r\geq0$, $\Delta_r(\mathcal L)$ is the sheaf associated to the presheaf of abelian groups, which assigns to an open subset $U\subset Y$ the free abelian group generated by elements of $\Gamma(U,\mathcal L^{r+1})$. For $r<0$, $\Delta_r(\mathcal L)=0$.
\begin{definition}\label{defn infty homotopy}
 An $\mathcal L$-indexed $\infty$-homotopy from $\mathcal F^*$ to $\mathcal G^*$ is a morphism of complexes of sheaves of abelian groups
 \begin{eqnarray}\label{Delta complex infty}
 	\mathrm{Ho}:\Delta_*(\mathcal L)\to  \mathcal Hom^{*}_{\mathcal O_Y}(\mathcal F^*,\mathcal G^*).
 \end{eqnarray}
 In other words, $\mathrm{Ho}$ is a family of morphisms
 	$$
 	\mathrm{Ho}^r:\mathcal L^{r+1}\to \mathcal Hom^{-r}_{\mathcal O_Y}(\mathcal F^*,\mathcal G^*),~r\geq0
 	$$
 	such that
 	$$
 	\delta\circ \mathrm{Ho}^r=d_{\mathcal Hom}\circ\mathrm{Ho}^{r+1},
 	$$
 	and the images of $\mathrm{Ho}^0$ are morphism of complexes. Here $\mathcal L^{r+1}$ is the direct product of $(r+1)$ copies of $\mathcal L$ and for $f:\mathcal L^{r+1}\to\mathcal Hom^{-r}_{\mathcal O_Y}(\mathcal F^*,\mathcal G^*)$,
 	$$
 	\delta f:\mathcal L^{r+2}\to\mathcal Hom^{-r}_{\mathcal O_Y}(\mathcal F^*,\mathcal G^*),~(l_0,\cdots,l_{r+1})\mapsto\sum_{q=0}^{r+1}(-1)^qf(\cdots,\widehat{l_q},\cdots).
 	$$
 \end{definition}

Now let $k$ be a perfect field of characteristic $p>0$, $X$ a smooth variety of dimension $n$ over $k$ and $D\subset X$ a reduced NCD. We assume the pair $(X,D)$ is $W_2(k)$-liftable. We choose and then fix such a lifting $(\tilde X,\tilde D)$. Set $\sX/\sS=((X,D)/k,(\tilde X',\tilde D')/W_2(k))$, where $(\tilde X',\tilde D')$ is the fiber product $(\tilde X,\tilde D)\times_{W_2(k),\sigma}W_2(k)$. Let $F: X\to X'$ be the relative Frobenius morphism over $k$. For simplicity, we write $\Omega_{X_{\log}/k}=\Omega_X(\log D)$. The notion of an $\sL$-indexed  $\infty$-homotopy comes from our attempt to generalize the construction of an \emph{explicit} quasi-isomorphism due to Deligne-Illusie \cite{DI} to coefficients.
\begin{theorem}[{\cite[Th\'eor\`eme 2.1]{DI}}]\label{DI case}
Let $\mathcal L=F_*\mathcal L_{\mathcal X/\mathcal S}$ be the sheaf of log Frobenius liftings. Then there is an explicit $\mathcal L$-indexed $\infty$-homotopy \begin{eqnarray*}\label{Ho for trivial}
 		\mathrm{Ho}:\Delta_*(\mathcal L)\to\mathcal Hom^*_{\mathcal O_{X'}}(\bigoplus_{i=0}^{p-1}\Omega^i_{X'_{\log}/k}[-i],\tau_{<p}F_*\Omega^*_{X_{\log}/k})
 	\end{eqnarray*}
 	such that $\mathrm{Ho}^0$ sends any section of $\mathcal L$ to a quasi-isomorphism.
 \end{theorem}
Ogus-Vologodsky \cite{OV} (later Schepler \cite{S} in the logarithmic setting) establishes a theory of coefficients in the context of nonabelian Hodge theory in positive characteristic. For a nilpotent Higgs module $(E,\theta)$ of level $\ell\leq p-1$ over $(X',D')/k$, there is a corresponding module with integrable connection over $(X,D)/k$
$$
(H,\nabla):=C^{-1}_{\sX/\sS}(E,\theta),
$$ 
the so-called inverse Cartier transform of $(E,\theta)$. For $(E,\theta)=(\sO_{X'},0)$, its inverse Cartier transform is nothing but $(\sO_X,d)$. Recall that for an integrable $\lambda$-connection $(E,\nabla)$ over $(X,D)/k$ ($\lambda\in k$), the associated de Rham complex $\Omega^*(E,\nabla)$ is defined as
$$
E\stackrel{\nabla}{\to} E\otimes \Omega_{X_{\log}/k}\stackrel{\nabla}{\to} E\otimes \Omega^2_{X_{\log}/k}\stackrel{\nabla}{\to}\cdots.
$$
The following result generalizes Theorem \ref{DI case}.
\begin{theorem}\label{infty homotopy}
Notation as above. Then there is an explicit $\sL:=F_*\mathcal L_{\mathcal X/\mathcal S}$-indexed $\infty$-homotopy $\mathrm{Ho}$ from $\tau_{<p-\ell}\Omega^{*}(E,\theta)$ to $\tau_{<p-\ell}F_{*}\Omega^{*}(H,\nabla)$
				$$
		\mathrm{Ho}^r:\mathcal L^{r+1}\to\mathcal Hom^{-r}_{\mathcal O_{X'}}(\tau_{<p-\ell}\Omega^{*}(E,\theta),\tau_{<p-\ell}F_{*}\Omega^{*}(H,\nabla))
		$$
		such that the images of $\mathrm{Ho}^0$ are quasi-isomorphisms.  
\end{theorem}
The above theorem is \cite[Theorem 3.5]{SZ}. The purpose of the note is to provide the omitted computational details, which are completely elementary, but unfortunately heavy in notations. We organize it as follows: \S1 contains our motivations for introducing the so-called Higgs-de Rham ring $B_{\mathrm{HdR}}$ and the associated two-dimensional discrete initial value problem over it. In \S2, we recall the construction of $B_{\mathrm{HdR}}$, and verify that the explicit formula (Definition \ref{varphi construction}) is indeed a solution to the initial value problem. This is the major part of the note. In \S3, we show that our construction is independent of the choices of frames of $\Omega^1_{X'_{\log}/k}$.

\section{Motivation}
Keep the notation as in Theorem \ref{infty homotopy}. For simplicity, we assume $\dim(X)<p-\ell$. Suppose that there is an $\mathcal L$-indexed $\infty$-homotopy of complexes of $\sO_{X'}$-modules
$$
\mathrm{Ho}:\Delta_*(\mathcal L)\to\mathcal Hom^{*}(\Omega^{*}(E,\theta),F_*\Omega^{*}(H,\nabla))
$$
such that 
$$
\mathrm{Ho}^0:\Delta_0(\mathcal L)\to\mathcal Hom^{0}(\Omega^{*}(E,\theta),F_*\Omega^{*}(H,\nabla))$$
is given by $\tilde F\mapsto\varphi_{\tilde F}$ and $\varphi_{\tilde F}$ is a quasi-isomorphism. Given a sequence of global liftings  $\tilde F_0,\cdots,\tilde F_r,\cdots$ of $F$. For any $r,s\geq 0$, consider 
$$
\mathrm{Ho}^r(\tilde F_0,\cdots,\tilde F_r)_s\in\mathrm{Hom}_{\mathcal O_{X'}}(\Omega^{r+s}(E,\theta),F_*\Omega^s(H,\nabla)).
$$
For any $q\geq0$, we set $\rho_q:\mathbb N\to\mathbb N-\{q\}$ to be the unique increasing bijection. Consider
$$
\mathrm{Ho}^r(\tilde F_{\rho_q(0)},\cdots,\tilde F_{\rho_q(r)})_s\in\mathrm{Hom}_{\mathcal O_{X'}}(\Omega^{r+s}(E,\theta),F_*\Omega^s(H,\nabla)).
$$
Since $\mathrm{Ho}$ is a morphism of complexes, we have
\begin{eqnarray}\label{verification infty homotopy}
\begin{array}{cl}
&\sum_{q=0}^{r}(-1)^q\mathrm{Ho}^{r-1}(\tilde F_{\rho_q(0)},\cdots,\tilde F_{\rho_q(r-1)})_{s+1}\\
=&(-1)^s\nabla\mathrm{Ho}^r(\tilde F_0,\cdots,\tilde F_r)_s+(-1)^{s+1}\mathrm{Ho}^r(\tilde F_0,\cdots,\tilde F_r)_{s+1}\theta.
\end{array}
\end{eqnarray}
Next, we try to understand the equality above.
\subsection{Understanding $\mathrm{Ho}^r(\tilde F_0,\cdots,\tilde F_r)_s$}
Assume there exists a basis for $\Omega_{X'_{\log}/k}$, say $\omega'_1,\cdots,\omega'_n$. Write $\theta=\sum_{i=1}^n\theta_i\otimes \omega'_i$. Take an element $x\in\Omega^{r+s}(E,\theta)$ which has skew-symmetric form
\begin{eqnarray}\label{skew-symmetric x}x=\frac{1}{(r+s)!}\sum_{i_{1},\cdots,i_{r+s}}e_{i_{1},\cdots,i_{r+s}}\otimes\omega_{i_{1}}\wedge \cdots\wedge\omega_{i_{r+s}},\end{eqnarray}
where $e_{i_{1},\cdots,i_{r+s}}\in E$ satisfies
$e_{i_{\sigma(1)},\cdots,i_{\sigma(r+s)}}=\mathrm{sgn}(\sigma)e_{i_{1},\cdots,i_{r+s}}$ for any permutation of $\{1,\cdots,r+s\}$. We hope that there is a universal polynomial $\varphi_p(r,s)$ over $\mathbb F_p$ of indeterminants 
\begin{eqnarray}\label{generators}
\begin{array}{c}
\zeta_{k,l},~0\leq k\leq r,1\leq l\leq n;~h_{k,l},~1\leq k\leq r,1\leq l\leq n;\\
\theta_i,~1\leq i\leq n;~e_I,~I\subset\{1,\cdots,n\},|I|=r+s
\end{array}
\end{eqnarray}

such that $\mathrm{Ho}^r(\tilde F_0,\cdots,\tilde F_r)_s(x)$ can be obtained by the following two steps:
\begin{itemize}
\item evaluating $\varphi_p(r,s)$ at 
$$
\zeta_{\tilde F_k}(\omega_l'),~h_{\tilde F_{k-1}\tilde F_k}(\omega_l'),~\theta_i,~e_{i_1,\cdots,i_{r+s}}~(I=\{i_1,\cdots,i_{r+s}\},~i_1<\cdots<i_{r+s}).
$$
which is a section of $E\otimes F_*\Omega^s_{X_{\log}/k}$;
\item using identifications
\begin{eqnarray}\label{identification F_0}
 E\otimes F_*\Omega^s_{X_{\log}/k}=F_*\Omega^s(H_{\tilde F_0},\nabla_{\tilde F_0})\cong F_*\Omega^s(H,\nabla),
\end{eqnarray}
we regard the section obtained above as a section of $F_*\Omega^s(H,\nabla)$.
\end{itemize}
It  is natural to think the indeterminants above are pairwisely commutative except for the skew-symmetric relation 
$$
\zeta_{k,l}\zeta_{k',l'}=-\zeta_{k',l'}\zeta_{k,l}.
$$
\subsection{Understanding $\mathrm{Ho}^r(\tilde F_{\rho_q(0)},\cdots,\tilde F_{\rho_q(r-1)})_{s+1}$}
Firstly, we evaluate $\varphi_p(r-1,s+1)$ at
$$
\zeta_{\tilde F_{\rho_q(k)}}(\omega_l),~h_{\tilde F_{\rho_q(k-1)}\tilde F_{\rho_q(k)}}(\omega_l),~\theta_i,~e_{i_1,\cdots,i_{r+s}}
$$
which is a section of $E\otimes F_*\Omega^{s+1}_{X_{\log}/k}$.
 Next, we regard this section as a section of $F_*\Omega^{s+1}(H,\nabla)$ via
 $$
 E\otimes F_*\Omega^{s+1}_{X_{\log}/k}=F_*\Omega^{s+1}(H_{\tilde F_{\rho_q(0)}},\nabla_{\tilde F_{\rho_q(0)}})\cong F_*\Omega^{s+1}(H,\nabla).
$$ 
Note that $\rho_q(0)=0$ for $q>0$ and $\rho_0(0)=1$ for $q=0$. Using the transition morphism $G_{\tilde F_0\tilde F_1}$, we have the following important diagram
$$
\xymatrix{
E\otimes F_*\Omega^\bullet_{X_{\log}/k}\ar[r]^-{=}\ar[d]^-{G_{\tilde F_0\tilde F_1}}&F_*\Omega^\bullet(H_{\tilde F_1},\nabla_{\tilde F_1})\ar[r]^-{\cong}&F_*\Omega^\bullet(H,\nabla)\ar[d]^-{=}\\
E\otimes F_*\Omega^\bullet_{X_{\log}/k}\ar[r]^-{=}&F_*\Omega^\bullet(H_{\tilde F_0},\nabla_{\tilde F_0})\ar[r]^-{\cong}&F_*\Omega^\bullet(H,\nabla).
}
$$
\subsection{Understanding $\nabla$}
Using the identifications \eqref{identification F_0}, we have
$$F_*\Omega^\bullet(H,\nabla)=E\otimes F_*\Omega^\bullet_{X_{\log}/k}$$
and 
$$\nabla=\nabla_{\mathrm{can}}+(\sum_{l=1}^n\theta_l\otimes\zeta_{\tilde F_0}(\omega_l))\wedge.$$
Here $\nabla_{\mathrm{can}}(e\otimes\omega)=e\otimes d\omega$ and 
$$
[(\sum_{l=1}^n\theta_l\otimes\zeta_{\tilde F_0}(\omega_l))\wedge](e\otimes\omega)=\sum_{l=1}^n\theta_l(e)\otimes[\zeta_{\tilde F_0}(\omega_l)\wedge\omega].
$$
\subsection{Understanding $\theta$}
Let $x\in\Omega^{r+s}(E,\theta)$ be given as \eqref{skew-symmetric x}. Then
\begin{eqnarray}\label{skew-symmetric theta e}\theta(x)=\frac{1}{(r+s+1)!}\sum_{i_{1},\cdots,i_{r+s+1}}\theta(x)_{i_{1},\cdots,i_{r+s+1}}\otimes\omega_{i_{1}}\wedge \cdots\wedge\omega_{i_{r+s+1}},\end{eqnarray}
where
$\theta(x)_{i_{1},\cdots,i_{r+s+1}}=\sum_{k=1}^{r+s+1}(-1)^{k-1}\theta_{i_{k}}e_{i_{1},\cdots,\hat{i_{k}},\cdots,i_{r+s+1}}$. One can check that for any permutation $\sigma$ of $\{1,\cdots,r+s+1\}$, we have
$$
\theta(x)_{i_{\sigma(1)},\cdots,i_{\sigma(r+s+1)}}=\mathrm{sgn}(\sigma)\theta(x)_{i_{1},\cdots,i_{r+s+1}}.
$$

The universality of $\varphi_p(r,s)$ implies that its formation should be independent of the characteristic $p$ and the dimension of $X$. In other words,  there should be a universal function $\varphi_\infty$ defined on $\mathbb Z\times\mathbb Z$ which takes value in a ring generated by \eqref{generators} over $\mathbb Q$. Moreover, $\varphi$ should satisfies the equation described by \eqref{verification infty homotopy}. This viewpoint will be achieved in the next subsection.
\subsection{An initial value problem and its solution}
Let $R$ be any commutative ring with identity. Let $A$ be the polynomial algebra over $R$ with the following three types of indeterminate: 
\begin{itemize}
\item $\theta _l, ~l=1,2,\cdots$;
\item $e_I,~I\subset \mathbb \Z_{>0},|I|<\infty$;
\item $h_{k,l},~k,l=1,2,\cdots$.
\end{itemize}
Let $\mathfrak a$ be the ideal generated by indeterminate of the second type, and let $\bar A$ be the quotient ring $A/\mathfrak a^2$. Let $M$ be the free $\bar A$-module generated by 
\begin{itemize}
\item $\zeta_{k,l},k=0,1,~\cdots,~l=1,2,\cdots$.
\end{itemize}
Recall that the exterior algebra $B_0:=\bigwedge_{\bar A}(M)$ is defined to be the quotient of the tensor algebra $T_{\bar A}(M)$ by the two-sided ideal generated by all expressions $x\otimes x$ for $x\in M$. So $B_0=\bigoplus_n\bigwedge^n(M)$ is a skew commutative graded $\bar A$-algebra. 
\begin{definition}
Let $\mathcal I\subset B_0$ be the two-sided ideal generated by $\{\theta_i,e_I,h_{k,l},\zeta_{k,l}\}$, and for $s\in \N$,
$\mathcal J_s$ be the two-sided ideal generated by $\theta_i,e_I,h_{k,l},\zeta_{k,l}$ with $i,|I|,k+l\geq s$. Let $B_1$ to be the completion of $B_0$ with respect to the decreasing family of two-sided ideals $\{\mathcal I^s+\mathcal J^s\}_{s}$. Finally, we define 
 $$B_{\mathrm{HdR}}:=B_1[\theta_i^{-1},i\geq1].$$
 We call it the Higgs-de Rham ring. To emphasize this ring is defined over $R$, we add a superscript $R$, i.e. $B^R_{\mathrm{HdR}}$. When $R=\mathbb Q$, we simplify write it as $B_{\mathrm{HdR}}$.
\end{definition}
 For each positive number $m$, let $\mathfrak b^R_m$ be the closure of the two-sided ideal  generated by monomials in $h_{k,l}, \zeta_{k,l}$ with total power $\geq m$. The quotient ring $B^R_{\mathrm{HdR},m}:=B^R_{\mathrm{HdR}}/\mathfrak{b}^R_m$ will be important for later use. Note that there is a natural decomposition of $R$-modules
\begin{eqnarray}\label{decomposition B}
B^R_{\mathrm{HdR}}=B^R_{\mathrm{HdR},m}\bigoplus\mathfrak{b}_m^R.
\end{eqnarray}
In fact, any element in $B^R_{\mathrm{HdR}}$ can be uniquely expressed as the sum of an  $R$-linear form combination of monomials in $\theta_l,e_I,h_{k,l},\zeta_{k,l}$ such that the total power in $h_{k,l},\zeta_{k,l}$ is $<m$ and an  $R$-linear form combination of monomials in $\theta_l,e_I,h_{k,l},\zeta_{k,l}$ such that the total power in $h_{k,l},\zeta_{k,l}$ is $\geq m$.
We introduce a subspace $B^R_{\mathrm{HdR},f}:=\bigcup_mB^R_{\mathrm{HdR},m}$. 

Next, our discussion is over $\mathbb Q$.
Let $d: B_0\to B_0$ be the additive map determined by the following rules:
$$
d\theta_l=de_I=d\zeta_{k,l}=0,~dh_{k,l}=\zeta_{k,l}-\zeta_{k-1,l},~d(xy)=(dx)y+(-1)^{n}xdy,
$$
for $x\in \bigwedge^n(M)$. Clearly, $d$ is continuous with the topology defined by $\{\mathcal I^s+\mathcal J^s\}_{s}$. Thus it extends uniquely an additive continuous operator on $B_1$, which is also denoted by $d$. We regard the element  $\sum_{l=1}^\infty\theta_l\zeta_{0,l}$ as an operator on $B_1$ by left multiplication. Set 
$$
\nabla=d+\sum_{l=1}^\infty\theta_l\zeta_{0,l}.
$$
Clearly, $\nabla^2=0$. 

Let $\Theta: B_0\to B_0$ be the differential operator on $B_0$ determined by 
$$
\Theta\theta_l=\Theta h_{k,l}=\Theta \zeta_{k,l}=0,
$$
$\Theta e_{\emptyset}=0$, and for $I\neq \emptyset$,
$$
\Theta e_I=\sum_{1\leq k\leq s}(-1)^{k-1}e_{\{\cdots,\hat{i_k},\cdots\}},~I=\{i_1,\cdots,i_s\},~ i_1<\cdots<i_s.
$$
It is easy to verify $\Theta^2=0$. Again, $\Theta$ extends uniquely a continuous operator over $B_1$, and it satisfies $\Theta^2=0$. 
For each $s\in \N$, we define a continuous additive operator $\mathrm{Shift}_s$ as follows: For $s=0$, $\mathrm{Shift}_0$ is determined by
$$
\theta_l\mapsto\theta_l,~e_I\mapsto e_I,~h_{k,l}\mapsto h_{k+1,l},~\zeta_{k,l}\mapsto\zeta_{k+1,l}.
$$
For $s>0$, $\mathrm{Shift}_s$ is determined by
$$
\theta_l\mapsto\theta_l,~e_I\mapsto e_I,~h_{k,l}\mapsto\left\{\begin{matrix}h_{k,l},~k<s\\
 h_{s,l}+h_{s+1,l},~k=s\\
 h_{k+1,l},~k>s
\end{matrix}\right.,~\zeta_{k,l}\mapsto\left\{\begin{matrix}
\zeta_{k,l},~k<s\\
\zeta_{k+1,l},~k\geq s.
\end{matrix}\right.
$$
 We linearly extend $\nabla, \Theta, \mathrm{Shift}_s, s\in \N$ to $B_{\mathrm{HdR}}$. The symbol $\mathrm{exp}(\sum_{l=0}^\infty\theta_l h_{1,l})$ is a well-defined element 
$$
1+\sum_{l=0}^\infty\theta_l h_{1,l}+\frac{1}{2}(\sum_{l=0}^\infty\theta_l h_{1,l})^2+\cdots
$$
in $B^{\Q}_{\mathrm{HdR}}$, and we regard it as an operator on $B^{\Q}_{\mathrm{HdR}}$ by left multiplication. Set 
$$
\delta_0:=\mathrm{exp}(\sum_{l=0}^\infty\theta_l h_{1,l})\mathrm{Shift}_0,\quad \delta_s:=\mathrm{Shift}_s, \ s>0. 
$$ 
Finally, we define the operator $D:\mathrm{Hom}_{\mathrm{Set}}(\mathbb Z\times\mathbb Z,B_{\mathrm{HdR}})\to\mathrm{Hom}_{\mathrm{Set}}(\mathbb Z\times\mathbb Z,B_{\mathrm{HdR}})$ by sending $\varphi=\{\varphi(r,s)\}\in \mathrm{Hom}_{\mathrm{Set}}(\mathbb Z\times\mathbb Z,B_{\mathrm{HdR}})$
to 
$$ 
(r,s)\mapsto(\nabla\varphi(r,s-1)+(-1)^{s}\sum^r_{k=0}(-1)^k\delta_k\varphi(r-1,s)-\Theta\varphi(r,s)).
$$
\begin{theorem}\label{solvability}
The following two-dimensional discrete initial value problem over $B_{\mathrm{HdR}}$ is solvable:
\begin{eqnarray}\label{initial problem}
\left\{\begin{matrix}
D\varphi=0,~\varphi\in\mathrm{Hom}_{\mathrm{Set}}(\mathbb Z\times\mathbb Z,B_{\mathrm{HdR}});\\
\varphi(0,s)=\sum_{I\subset\mathbb N_{>0},|I|=s}e_I\zeta_{0,I},~s\geq0;\\
\varphi(r,s)=0,~r<0~\mathrm{or}~s<0.
\end{matrix}\right.
\end{eqnarray}
Here $\zeta_{0,\emptyset}=1$, and $\zeta_{0,I}:=\zeta_{0,i_1}\cdots\zeta_{0,i_s}$ for $I=\{i_1,\cdots,i_s\}$ with $i_1<\cdots<i_s$.   
\end{theorem}

\section{The proof of Theorem \ref{solvability}}
First we make a table of notations as follows:\\

\noindent {\bf Notation:}
\begin{itemize}
\item[$\overline i$\ ] $\{i^k\}_{k\in\mathbb Z_{>0}}$ such that $i^k\geq0$ and there are only finitely many $k\in\mathbb Z_{>0}$ such that $i^k>0$.

\item[$\underline i$\ ] $\{i_l\}_{l\in\mathbb Z_{>0}}$ such that $i_l\geq0$ and there are only finitely many $l\in\mathbb Z_{>0}$ such that $i_l>0$.

\item[$\underline{\underline j}$\ ] It is a family $\{j_{k,l}\}_{(k,l)\in\mathbb Z_{>0}\times\mathbb Z_{>0}}$ such that $j_{k,l}\geq0$ holds for any pair $(k,l)$ and $j_{k,l}>0$ holds for only finitely many pair $(k,l)$.
\item[$\underline s$\ ] $\{s_l\}_{l\in\mathbb Z_{\geq0}}$ such that each $s_l\geq0$ and $s_l>0$ holds only for finitely many $l$.

\item[$\eta^k_l$\ ] We set $\eta^k_l:=1$ if $k=l$ and $0$ otherwise.

\item[$\underline \eta^k$\ ]  $\{\eta_l^k\}_{l\in\mathbb Z_{\geq0}}$.

\item[$\underline{\underline{j}}_{\overline{i}}\ $] $\overline{\underline{j}}_{\overline{i}}:=\prod_kj_{k,i^k}$,  here $j_{k,0}:=1$.

\item[$\theta^{\underline{\underline j}}\ $] $\prod_{l=1}^{\infty}\theta_{l}^{\sum_k j_{k,l}}$.

\item[$\theta_{\overline i}$\ ] $\theta_{\overline i}:=\prod_k\theta_{i^k}$, here $\theta_0:=1$.

\item[$h^{[\underline{\underline j}]}$\ ] $h^{[\underline{\underline j}]}:=\prod_{k,l}\frac{h_{k,l}^{j_{k,l}}}{j_{k,l}!}$.

\item[$\cup\underline i$\ ] $\cup\underline i=\{i_l>0|l\in\mathbb Z_{>0}\}$.

\item[$\cup_q\underline{\underline j}$\ ] For any $q>0$,  we set $\cup_q\underline{\underline j}:=\{\tilde{j}_{k,l}\}$ and 
$$
\tilde{j}_{k,l}=\left\{
\begin{array}{cc}
j_{k,l},&k<q,\\
j_{q,l}+j_{q+1,l},&k=q,\\
j_{k+1,l},&k>q.
\end{array}
\right.
$$

\item[$e_{\underline i,\overline i}$\ ] Let $\underline I:=\{i_l>0|l\in\mathbb Z_{>0}\}$ and $\overline I:=\{i^k>0|k\in\mathbb Z_{>0}\}$. Assume $\underline I=\overline I=\emptyset$. Set $e_{\underline i,\overline i}:=e_\emptyset$. Assume $\underline I\neq\emptyset,~\overline I\neq\emptyset$. Set
$$
\underline I=\{i_{l_1},\cdots,i_{l_s}\}~\mathrm{with}~l_1<\cdots<l_s,~\overline I=\{i^{k_1},\cdots,i^{k_r}\}~\mathrm{with}~k_1<\cdots<k_r.
$$
If $|\underline I\cup\overline I|<r+s$, then set $e_{\underline i,\overline i}:=0$. If $|\underline I\cup\overline I|=r+s$, then set
$$
e_{\underline i,\overline i}:=\mathrm{sgn}(\sigma)e_{\underline I\cup\overline I},
$$
where $\sigma$ is a permutation of $\{1,\cdots,r+s\}$ such that
$$
i'_{\sigma(1)}<\cdots<i'_{\sigma(r+s)},~(i'_1,\dots,i'_{r+s}):=(i_{l_1},\cdots,i_{l_s},i^{k_1},\cdots,i^{k_r}).$$
The remaining two cases $\underline I\neq\emptyset,\overline I=\emptyset$ and $\underline I=\emptyset,\overline I\neq\emptyset$ can be discussed similarly.

\item[$c(i,\underline{i})\ $] Assume $i\in\cup\underline i$. Set $c(i,\underline{i}):=l_0-1$, where $l_0$ is the smallest $l$ such that $i=i_l$.

\item[$\underline{\underline j}(p,q)\ $] $\{j_{k,l}\}$ such that $j_{k,l}=1$ if $(k,l)=(p,q)$ and $0$ otherwise.

\item[$a(\underline{\underline j},\underline s)$] $\frac{\prod_ls_l!}{\prod^\infty_{p=1}\prod^{s_{p-1}}_{q=0}\mathrm{max}\{1,q+\sum_{k\geq p,l}j_{k,l}+\sum_{l\geq p}s_l\}}$.

\item[$\mathbb Z^{\underline s}_{>0,\uparrow}$\ ] Let $s:=\sum_l s_l>0$. If $s=0$, then put $\mathbb Z^{\underline s}_{>0,\uparrow}$ to be the set of single sequence $\underline 0=\{0\}_{l\in\mathbb Z_{>0}}$. Assume $s>0$.
Let $\{l|s_l>0\}=\{l_1,\cdots,l_q\}$ with $l_1<\cdots<l_q$. Put $\mathbb Z^{\underline s}_{>0,\uparrow}$ to be the set of sequences
\begin{eqnarray}\label{underline i}
i_{l_1,1},\cdots,i_{l_1,s_{l_1}},\cdots,i_{l_q,1},\cdots,i_{l_q,s_{l_q}},0,\cdots,0,\cdots
\end{eqnarray}
satisfying $i_{l_1,1}<\cdots<i_{l_1,s_{l_1}},\cdots,i_{l_q,1}<\cdots<i_{l_q,s_{l_q}}$.
\item[$\zeta_{\underline s,\underline i}$\ ] Let $s=\sum_{l\geq 0}s_l$. If $s=0$, we set $\zeta_{\underline s,\underline i}:=1$. Assume $s>0$. Let $\{l|s_l>0\}=\{l_1,\cdots,l_q\}$ with $l_1<\cdots<l_q$ and let 
$$\underline i=\{i_{l_1,1},\cdots,i_{l_1,s_{l_1}},\cdots,i_{l_q,1},\cdots,i_{l_q,s_{l_q}},0,\cdots,0,\cdots\}\in\mathbb Z^{\underline s}_{>0,\uparrow}.$$ 
Set 
$$\zeta_{\underline s,\underline i}:=\zeta_{l_1,i_{l_1,1}}\cdots\zeta_{l_1,i_{l_1,s_{l_1}}}\cdots\zeta_{l_q,i_{l_q,1}}\cdots\zeta_{l_q,i_{l_q,s_{l_q}}}.$$

\item[$T(r)$\ ] $\{i^k\}_{k\in\mathbb Z_{>0}}$ such that $i^k>0$ for $k\leq r$ and $i^k=0$ for $k>r$.
\item[$T(r,s)$\ ] For any $r,s\geq0$, let $T(r,s)$ be the set consisting of triples $(\underline{\underline j},\underline s,\underline i)$ satisfying the following conditions:
\begin{itemize}
\item $j_{k,l}=s_k=0$ for $k>r$;
\item $\underline i\in\mathbb Z^{\underline s}_{>0,\uparrow}$.
\end{itemize}

\item[$T_q(r,s)$\ ] The subset of $T(r,s)$ consisting of $(\underline{\underline j},\underline s,\underline i)$ such that $\sum_{k,l}j_{k,l}+\sum_ls_l<q$.

\item[$\mathrm{Del}_q(\underline{\underline{j}})\ $] For $q\geq 1$, we set $\mathrm{Del}_q(\underline{\underline{j}}):=\{\tilde{j}_{k,l}\}$. Here $\tilde{j}_{k,l}=j_{k,l}$ for $k<q$ and $\tilde{j}_{k,l}=j_{k+1,l}$ for $k\geq q$.

\item[$\mathrm{Del}_q(\underline s)$\ ] For any $q\geq0$, we set $\mathrm{Del}_q(\underline s):=\{\tilde s_k\}$ to be the sequence obtained from $\underline s$ by deleting the term $s_q$. More precisely, we have $\tilde s_k=s_k$ for $k<q$ and $\tilde s_k=s_{k+1}$ for $k\geq0$.

\item[$\mathrm{Del}^q(\overline i)$\ ] The sequence $i^1\cdots,\widehat{i^q},\cdots,i^r,\cdots$ which is obtained from $\overline i$ by deleting the $q$-th term $i^q$.

\item[$\mathrm{Del}(i,\underline i)$\ ] Assume $i\in\cup\underline i$. Set $\mathrm{Del}(i,\underline i):=\mathrm{Del}_{c(i,\underline i)+1}(\underline i)$.
\end{itemize}
Next, we provide our solution to  \eqref{initial problem} as follows:
\begin{definition}\label{varphi construction}
For any $r,s\geq0$, we define
$$
\varphi_\infty(r,s):=\sum_{(\underline{\underline j},\underline s,\underline i)\in T(r,s)}[\sum_{\overline i\in T(r)}a(\underline{\underline j},\underline s)\underline{\underline j}_{\overline i}\theta^{-1}_{\overline i}\theta^{\underline{\underline j}}e_{\underline i,\overline i}]h^{[\underline{\underline j}]}\zeta_{\underline s,\underline i}.
$$
\end{definition}
The verification that $\varphi_\infty$ satisfies \eqref{initial problem} boils down to the following family of equalities:
$$
\nabla\varphi_\infty(r,s)+(-1)^{s+1}\sum^r_{k=0}(-1)^k\delta_k\varphi_\infty(r-1,s+1)=\Theta\varphi_\infty(r,s+1),~r\geq0,~s\geq-1.
$$

Let $V$ be the $\mathbb Q$-vector subspace of $B_{\mathrm{HdR}}$ generated by the non-zero monomials
$h^{[\underline{\underline j}]}\zeta_{\underline s,\underline i}$. Let $S$ be the closure of the subring of $B_{\mathrm{HdR}}$ generated by indeterminants  $\theta_l,e_I$. Define a pairing
$$
<,>:V\otimes_\mathbb Q V\to\mathbb Q
$$
by sending a non-zero $h^{ [\underline{\underline j}]}\zeta_{\underline s,\underline i}\otimes h^{[\underline{\underline j'}]}\zeta_{\underline s',\underline i'}$ to $1$ if $\underline{\underline j}=\underline{\underline j'},\underline s=\underline s',\underline i=\underline i'$ and $0$ otherwise. It is obvious that this pairing is perfect. Let $<,>_S$ be the linear extension of $<,>$ to $V_S$. Observe that any element in $ B_{\mathrm{HdR}}$ can be uniquely written as
$$
\sum b^{\underline{\underline j}}_{\underline s,\underline i}h^{[\underline{\underline j}]}\zeta_{\underline s,\underline i},~b^{\underline{\underline j}}_{\underline s,\underline i}\in S.
$$
This observation allows us to define an $S$-linear pairing
$$
<,>_B:B_{\mathrm{HdR}}\otimes_SV_S\to S,~(\sum b^{\underline{\underline j}}_{\underline s,\underline i}h^{[\underline{\underline j}]}\zeta_{\underline s,\underline i})\otimes v\mapsto\sum b^{\underline j}_{\underline s,\underline i}<h^{[\underline{\underline j}]}\zeta_{\underline s,\underline i},v>_S.
$$
Note that this pairing induces an injective $\mathbb Q$-linear map
\begin{eqnarray}\label{pairing}
B_{\mathrm{HdR}}\to\mathrm{Hom}_\mathbb Q(V,S),~b\mapsto<b,->_B.
\end{eqnarray}

Thanks to the lemma below and \eqref{pairing} is injective, the equalities above follows from the coincidence of 
\begin{eqnarray}\label{nabla LHS}
\quad\quad<\varphi_\infty(r,s),\nabla^*(h^{[\underline{\underline j}]}\zeta_{\underline s,\underline i})>+<(-1)^{s+1}\varphi_\infty(r-1,s+1),\sum^r_{k=0}(-1)^k\delta_k^*(h^{[\underline{\underline j}]}\zeta_{\underline s,\underline i})>
\end{eqnarray}
and
\begin{eqnarray}\label{Theta RHS}
<\Theta\varphi_\infty(r,s+1),h^{[\underline{\underline j}]}\zeta_{\underline s,\underline i}>
\end{eqnarray}
for any $r\geq0,s\geq-1$ and any non-zero $h^{[\underline{\underline j}]}\zeta_{\underline s,\underline i}$. Clearly, it suffices to check that the expressions above coincide for $h^{[\underline{\underline j}]}\zeta_{\underline s,\underline i},(\underline{\underline j},\underline s,\underline i)\in T(r,s+1)$.

\begin{lemma}\label{adjoint}
The operators $\nabla,\delta_k$ on $B_{\mathrm{HdR}}$ have respective adjoints $\nabla^*,\delta_k^*$ on $V_S$ with respect to the pairing $<,>_B$. 
\end{lemma}
\begin{proof}
Let us construct $\nabla^*$. Given any non-zero $h^{[\underline{\underline j}]}\zeta_{\underline s,\underline i}\in V$. For any $k\geq0$, set $\nabla^*_k(h^{[\underline{\underline j}]}\zeta_{\underline s,\underline i})$ to be
$$
\left\{
\begin{matrix}
\sum_{q=1}^{s_0}(-1)^{c(i_{0,q},\underline i)}[\theta_{i_{0,q}}h^{[\underline{\underline j}]}\zeta_{\underline s-\underline\eta^0,\mathrm{Del}(i_{0,q},\underline i)}-h^{[\underline{\underline j}+\underline{\underline j}(1,{i_{0,q}})]}\zeta_{\underline s-\underline\eta^0,\mathrm{Del}(i_{0,q},\underline i)}],&k=0,s_0>0,\\
\sum_{q=1}^{s_k}(-1)^{c(i_{k,q},\underline i)}[h^{[\underline{\underline j}+\underline{\underline j}(k,i_{k,q})]}\zeta_{\underline s-\underline\eta^k,\mathrm{Del}(i_{k,q},\underline i)}-h^{[\underline{\underline j}+\underline{\underline j}(k+1,{i_{k,q}})]}\zeta_{\underline s-\underline\eta^k,\mathrm{Del}(i_{k,q},\underline i)}],&k>0,s_k>0,\\
0,&s_k=0.
\end{matrix}
\right.
$$
Put $\nabla^*:=\sum_{k=0}^\infty\nabla^*_k$. By the definition of $\nabla$, one can easily check that $\nabla^*$ is the unique adjoint of $\nabla$ with respect to $<,>_B$.

Turn to the construction of $\delta_k^*$.  Given any non-zero $h^{[\underline{\underline j}]}\zeta_{\underline s,\underline i}\in V$. Set $\delta^*_k(h^{[\underline j]}\zeta_{\underline s,\underline i}):=0$ if $s_k\neq0$ and
$$
\delta^*_k(h^{[\underline j]}\zeta_{\underline s,\underline i}):=\left\{
\begin{matrix}
\theta^{\underline j_1}h^{[\mathrm{Del}_1(\underline{\underline j})]}\zeta_{\mathrm{Del}_0(\underline s),\underline i},&\mathrm{if}~k=0,s_0=0,\\
h^{[\cup_k\underline{\underline j}]}\zeta_{\mathrm{Del}_k(\underline s),\underline i},&\mathrm{if}~k>0,s_k=0.
\end{matrix}
\right.
$$
Here $\theta^{\underline j_1}=\prod_{l=1}^\infty\theta_l^{j_{1,l}}$.
By the construction of $\delta_k$, one can check that $\delta_k^*$ is the unique adjoint of $\delta_k$ with respect to $<,>_B$. 
\end{proof}

\begin{lemma}
Keep the notation above. We additionally assume that $r>0,s\geq0$. Then for any $0\leq k\leq r$ and any non-zero $h^{[\underline{\underline j}]}\zeta_{\underline s,\underline i}, (\underline{\underline j},\underline s,\underline i)\in T(r,s+1)$, we have
$$
\begin{array}{rcl}
M_k&:=&<\varphi(r,s),\nabla_k^*(h^{[\underline{\underline j}]}\zeta_{\underline s,\underline i})>+<(-1)^{s+1+q}\varphi(r-1,s+1),\delta_k^*(h^{[\underline{\underline j}]}\zeta_{\underline s,\underline i})>\\
&=&A_k+B_k^{+}+B_k^{-},\\
\\
A_k&:=&\left\{\begin{matrix}
\sum_{\overline i\in T(r),1\leq q\leq s_k} a(\underline{\underline j},\underline s)\underline{\underline j}_{\overline i}\theta^{-1}_{\overline i}\theta^{\underline{\underline j}}(-1)^{c(i_{k,q},\underline i)}\theta_{i_{k,q}}e_{\mathrm{Del}(i_{k,q},\underline i),\overline i},&\mathrm{if}~s_k>0;\\
0,&\mathrm{if}~s_k=0,
\end{matrix}
\right.\\
\\
B_k^{-}&:=&\left\{
\begin{array}{l}
0,~\mathrm{if}~k=0;\\
\sum_{\overline i\in T(r-1)}(-1)^{s+1+k}\mathrm{max}\{1-j_k,\sum_{k'\geq k+1}j_{k'}+\sum_{k'\geq k}s_{k'}\}a(\underline{\underline j},\underline s)\mathrm{Del}_k(\underline{\underline j})_{\overline i}\theta^{-1}_{\overline i}\theta^{\underline{\underline j}}e_{\underline i,\overline i},\\
\mathrm{if}~0<k\leq r;
\end{array}
\right.\\
\\
B_k^{+}&:=&\left\{
\begin{array}{l}
\sum\limits_{\overline i\in T(r-1)}(-1)^{s+1+k}\mathrm{max}\{1,\sum\limits_{k'\geq k+1}j_{k'}+\sum\limits_{k'\geq k+1}s_{k'}\}a(\underline{\underline j},\underline s)\mathrm{Del}_{k+1}(\underline{\underline j})_{\overline i}\theta^{-1}_{\overline i}\theta^{\underline{\underline j}}e_{\underline i,\overline i},\\
\mathrm{if}~0\leq k<r;\\
0,~\mathrm{if}~k=r.
\end{array}
\right.
\end{array}
$$
Here $\underline i$ is written as \eqref{underline i} and $j_{k'}:=\sum_{l\geq1}j_{k',l}$. For other notations involved in the expressions above, the reader may refer to the notation table located at the beginning of this subsection.
\end{lemma}
\begin{proof}
This lemma can be checked case by case. (i) The case of $0<k<r,s_k>0$. By the proof Lemma \ref{adjoint}, $M_k$ equals the difference between 
$$
\sum_{q=1}^{s_k}\sum_{\overline i\in T(r)}(-1)^{c(i_{k,q},\underline i)}a(\underline{\underline j}+\underline{\underline j}(k,i_{k,q}),\underline s-\underline\eta^l)(\underline{\underline j}+\underline{\underline j}(k,i_{k,q}))_{\overline i}\theta^{-1}_{\overline i}\theta^{\underline{\underline j}+\underline{\underline j}(l,i_{k,q})}e_{\mathrm{Del}(i_{k,q},\underline i),\overline i}
$$
and
$$
\sum_{q=1}^{s_k}\sum_{\overline i\in T(r)}(-1)^{c(i_{k,q},\underline i)}a(\underline{\underline j}+\underline{\underline j}(k+1,i_{k,q}),\underline s-\underline\eta^k)(\underline{\underline j}+\underline{\underline j}(k+1,i_{k,q}))_{\overline i}\theta^{-1}_{\overline i}\theta^{\underline{\underline j}+\underline{\underline j}(k+1,i_{k,q})}e_{\mathrm{Del}(i_{k,q},\underline i),\overline i}.
$$
The following facts can be checked directly, we omit the details:
\begin{itemize}
\item $a(\underline{\underline j}+\underline{\underline j}(k,i_{k,q}),\underline s-\underline\eta^k)=\mathrm{max}\{1,\sum_{k'\geq k+1}j_{k'}+\sum_{k'\geq k}s_{k'}\}s_k^{-1}a(\underline{\underline j},\underline s)$;
\item $a(\underline{\underline j}+\underline{\underline j}(k+1,i_{k,q}),\underline s-\underline\eta^k)=\mathrm{max}\{1,\sum_{k'\geq k+1}j_{k'}+\sum_{k'\geq k+1}s_{k'}\}s_k^{-1}a(\underline{\underline j},\underline s)$;
\item $[a(\underline{\underline j}+\underline{\underline j}(k,i_{k,q}),\underline s-\underline\eta^k)-a(\underline{\underline j}+\underline{\underline j}(k+1,i_{k,q}),\underline s-\underline\eta^k)]\underline{\underline j}_{\overline i}=a(\underline{\underline j},\underline s)\underline{\underline j}_{\overline i}$;
\item let $\overline i\in T(r)$, then $(\underline{\underline j}+\underline{\underline j}(k,i_{k,q}))_{\overline i}=\underline{\underline j}_{\overline i}+\eta_{i_{k,q}}^{i^k}\prod_{k'\neq k}j_{k',{i^{k'}}}$  ($j_{k',0}:=1$) and 
$$(\underline{\underline j}+\underline{\underline j}(k+1,i_{k,q}))_{\overline i}=\underline{\underline j}_{\overline i}+\eta_{i_{k,q}}^{i^{k+1}}\prod_{k'\neq k+1}j_{k',i^{k'}};$$
\item $\theta^{\underline{\underline j}+\underline{\underline j}(k,i_{k,q})}=\theta^{\underline{\underline j}+\underline{\underline j}(k+1,i_{k,q})}=\theta^{\underline{\underline j}}\theta_{i_{k,q}}$.
\end{itemize}
Using the facts above, $M_k$ can be expressed as the sum of $A_k$,
$$
\begin{matrix}
\sum_{q=1}^{s_k}\sum_{\overline i \in T(r),i^k=i_{k,q}}(-1)^{c(i_{k,q},\underline i)}(\sum_{k'\geq k+1}j_{k'}+\sum_{k'\geq k}s_{k'})s_k^{-1}\\
\cdot a(\underline{\underline j},\underline s)(\prod_{k'\neq k}j_{k',i^{k'}})\theta^{-1}_{\overline i}\theta^{\underline{\underline j}}\theta_{i_{k,q}}e_{\mathrm{Del}(i_{k,q},\underline i),\overline i}
\end{matrix}
$$
and
$$
\begin{matrix}
\sum_{q=1}^{s_k}\sum_{\overline i\in T(r),i^{k+1}=i_{k,q}}(-1)^{c(i_{k,q},\underline i)}\mathrm{max}\{1,\sum_{k'\geq k+1}j_{k'}+\sum_{k'\geq k+1}s_{k'}\}s_k^{-1}\\
\cdot a(\underline{\underline j},\underline s)(\prod_{k'\neq k+1}j_{k',i^{k'}})\theta^{-1}_{\overline i}\theta^{\underline{\underline j}}\theta_{i_{k,q}}e_{\mathrm{Del}(i_{k,q},\underline i),\overline i}.
\end{matrix}
$$
Replacing the sequences $\overline i\in T(r),i^k=i_{k,q}$ and $\overline i\in T(r),i^{k+1}=i_{k,q}$ by the subsequences obtained by respectively deleting $i^k$ and $i^{k+1}$, then we see that $M_k=A_k+B_k^-+B_k^+$.

The case of $0<k<r$ and $s_k=0$. By the proof of Lemma \ref{adjoint}, we have
$$
M_k=\sum_{\overline i\in T(r-1)}(-1)^{s+1+k}a(\cup_k\underline{\underline j},\mathrm{Del}_k(\underline s))(\cup_k\underline{\underline j})_{\overline i}\theta^{-1}_{\overline i}\theta^{\underline{\underline j}}e_{\underline i,\overline i}.
$$
Note that
$$
a(\cup_k\underline{\underline j},\mathrm{Del}_k(\underline s))=(\sum_{k'\geq k+1}j_{k'}+\sum_{k'\geq k}s_{k'})a(\underline{\underline j},\underline s)=(\sum_{k'\geq k+1}j_{k'}+\sum_{k'\geq k+1}s_{k'})a(\underline{\underline j},\underline s)$$
and $
(\cup_k\underline{\underline j})_{\overline i}=\mathrm{Del}_k(\underline{\underline j})_{\overline i}+\mathrm{Del}_{k+1}(\underline{\underline j})_{\overline i}$. Clearly, $M_k=B_k^-+B_k^+$. On the other hand, by definition $A_k=0$, from which we have $M_k=A_k+B_k^-+B_k^+$.

(iii) The case of $k=0$ can be checked similarly as above. When $k=r$,  there are three situations should be treated separately: (1) $s_r>0$, (2) $j_r>0,s_r=0$  and (3) $j_r=s_r=0$. We omit the details. This completes the proof.
\end{proof}

{\itshape The proof of Theorem \ref{solvability}.} We check the coincidence of \eqref{nabla LHS} and \eqref{Theta RHS} when $r>0,s\geq0$. The remain cases can be checked similarly. Given any non-zero $h^{[\underline j]}\zeta_{\underline s,\underline i}\in T(r,s+1)$. Note that 
\begin{eqnarray}\label{expression A}
\sum_{k=1}^rA_k=\sum_{\overline i\in T(r)} a(\underline{\underline j},\underline s)\underline{\underline j}_{\overline i}\theta^{-1}_{\overline i}\theta^{\underline{\underline j}}\sum_{i\in\cup\underline i}(-1)^{c(i,\underline i)}\theta_ie_{\mathrm{Del}(i,\underline i),\overline i}.
\end{eqnarray}
 On the other hand, we have
\begin{eqnarray}\label{expression B}
\begin{array}{c}
\sum_{k=0}^r(B_k^{-}+B_k^{+})=\sum_{k=0}^{r-1}(B_k^{+}+B_{k+1}^{-})\\
=\sum^{r-1}_{k=0}\sum_{\overline i\in T(r-1)}(-1)^{s+1+k}a(\underline{\underline j},\underline s)j_{k+1}\mathrm{Del}_{k+1}(\underline{\underline j})_{\overline i}\theta^{-1}_{\overline i}\theta^{\underline{\underline j}}e_{\underline i,\overline i}\\
=\sum_{\overline i\in T(r)}a(\underline{\underline j},\underline s)\underline{\underline j}_{\overline i}\theta^{-1}_{\overline i}\theta^{\underline{\underline j}}\sum_{k=0}^{r-1}(-1)^{s+1+k}\theta_{i^{k+1}}e_{\underline i,\mathrm{Del}_{k+1}(\overline i)}.
\end{array}
\end{eqnarray}
We point out that though the definition of $B_r^-$ is somehow complicated, the equality
$$
B^+_{r-1}+B^-_r=\sum_{\overline i\in T(r-1)}(-1)^{s+r}a(\underline{\underline j},\underline s)j_r\mathrm{Del}_r(\underline{\underline j})_{\overline i}\theta^{-1}_{\overline i}\theta^{\underline{\underline j}}e_{\underline i,\overline i}
$$
always holds true.
Combing \eqref{expression A}, \eqref{expression B}, $
\eqref{nabla LHS}=\sum_{k=0}^rM_k$
and the definition of $\Theta$, the coincidence of \eqref{nabla LHS} and  \eqref{Theta RHS} follows. This completes the proof of Theorem \ref{solvability}.
\subsection{$\varphi_\infty$ is a geometric solution}

 Return to the setting of \S2. To introduce the evaluation morphism in Definition \ref{evaluation}, we need a technique called the scalarization of nilpotent Higgs sheaves. Let $(E,\theta)$ be a nilpotent Higgs sheaf on $X_{\log}'/k$ of level $\leq \ell$. Define an $\mathcal O_{X'}$-module
$$R(E,\theta):=E\oplus\mathcal O_{X'}\oplus\rho_\theta(S^+ T_{X_{\log}'/k}),$$
where $S^+ T_{X_{\log}'/k}=\bigoplus_{i>0}S^iT_{X_{\log}'/k}$ and $\rho_\theta:S^\bullet T_{X_{\log}'/k}\to\mathcal End_{\mathcal O_{X'}}(E)$ is induced by $\theta$. It is easy to see that the canonical $\mathcal O_{X'}$-algebra structure on $E\oplus S^\bullet T_{X_{\log}'/k}$ induces an $\mathcal O_{X'}$-algebraic structure on $\mathrm{Sc}(E,\theta)$. Note that 
$$\theta\in\rho_\theta(S^+ T_{X_{\log}'/k})\otimes\Omega^1_{X_{\log}'/k}\subset\mathcal End_{\mathcal O_{X'}}(E)\otimes\Omega^1_{X_{\log}'/k},$$
then by construction we have
$$\theta\in R(E,\theta)\otimes\Omega^1_{X_{\log}'/k}.$$
Consequently, there is a Higgs module structure on $R(E,\theta)$:
$$R(E,\theta)\to R(E,\theta)\otimes\Omega^1_{X_{\log}'/k},~r\mapsto r\theta,~r\in R(E,\theta).$$
By abuse of notation, we denote this Higgs field by $\theta$. 
\begin{definition}\label{scalarization}
We call $\mathrm{Sc}(E,\theta):=(R(E,\theta),\theta)$ the scalarization of $(E,\theta)$.
\end{definition}
 
For simplicity, we abbreviate $R(E,\theta)$ as $R$.  Let $\mathrm{Frame}$ be the sheaf of frames of $\Omega^1_{X'_{\log}/k}$. Over any open subset $U$ of $X$ and any $r\geq0$, we give the following data:
 \begin{itemize}
 \item any frame $(\omega'_1,\cdots,\omega'_n)\in\mathrm{Frame}$;
 \item any $r+1$ liftings of $F$ over $U$, say $\tilde F_0,\cdots,\tilde F_r$;
 \item any section of $\tau_{<p-\ell}\Omega^\bullet(E,\theta)$ over $U'$ which can be written as
 $$
 e+\sum_{0<q<p-\ell}\frac{1}{q!}\sum_{1\leq i_1,\cdots,i_q\leq n}e_{i_1,\cdots,i_q}\otimes\omega_{i_1}'\wedge\cdots\wedge\omega_{i_q}'\in\Omega^\bullet(E,\theta)_{U'},~,e,e_{i_1,\cdots,i_q}\in E.
 $$
 \end{itemize}
 Using these data above and the tensor product ring structure on $R\otimes\Omega^\bullet_{X_{\log}/k}$,  one can construct a ring homomorphism
 $$
 B^{\mathbb Z_{(p)}}_{\mathrm{HdR},f}\to\Gamma(U',R\otimes F_*\Omega^\bullet_{X_{\log}/k})$$
 as follows: 
 \begin{itemize}
 \item $\theta_i\mapsto\vartheta_i\otimes 1\in\Gamma(U',\rho_\theta(S^+T_{X_{\log}/k})\otimes F_*\mathcal O_X)$ for $i\leq n$ and $\theta_i\mapsto0$ for $i>n$, where $\theta=\sum_{i=1}^n\vartheta_i\otimes\omega_i'$;
 \item $e_\emptyset\mapsto e\otimes1\in\Gamma(U',E\otimes F_*\mathcal O_X)$, $e_I\mapsto e_{i_1,\cdots,i_q}\otimes1\in\Gamma(U',E\otimes F_*\mathcal O_X)$ for $I=\{i_1,\cdots,i_q\}$ with $i_1<\cdots<i_q\leq n$ and $e_I\mapsto 0$ otherwise;
 \item $h_{k,l}\mapsto 1\otimes h_{\tilde F_k\tilde F_{k-1}}(\omega_l)\in\Gamma(U',\mathcal O_{X'}\otimes F_*\mathcal O_X)$ for $k\leq r$ and $l\leq n$ and $0$ otherwise;
 \item $\zeta_{k,l}\mapsto 1\otimes\zeta_{\tilde F_k}(\omega_l)\in\Gamma(U',\mathcal O_{X'}\otimes F_*\Omega^1_{X_{\log}/k})$ for $k\leq r$ and $l\leq n$ and $0$ otherwise.
 \end{itemize}
 We regard $B^{\mathbb Z_{(p)}}_{\mathrm{HdR},f}$ as a constant sheaf on $X'$. Clearly,  the construction above gives rise to a morphism of sheaves on $X'$
  $$\mathrm{Ev}:\mathrm{Frame}\times\bigsqcup_{r\geq0}\mathcal L^{r+1}\times\tau_{<p-l}\Omega^\bullet(E,\theta)\to\mathcal Hom(B^{\mathbb Z_{(p)}}_{\mathrm{HdR},f},R\otimes F_*\Omega^\bullet_{X_{\log}/k}).$$
 Set $(H,\nabla):=C^{-1}_{\mathcal X/\mathcal S}(E,\theta)$. Using the natural projection $R\otimes F_*\Omega^\bullet_{X_{\log}/k}\to E\otimes  F_*\Omega^\bullet_{X_{\log}/k}$ and the identifications
$$
E_{U'}\otimes F_*\Omega^\bullet_{U_{\log}/k}=\Omega^\bullet(H_{\tilde F_0},\nabla_{\tilde F_0})\cong F_*\Omega^\bullet(H,\nabla)_U,$$ one obtains the desired 
 \begin{definition}\label{evaluation}
The discussion above gives rise to an evaluation morphism
 $$
\mathrm{ev}:\mathrm{Frame}\times\bigsqcup_{r\geq0}\mathcal L^{r+1}\times\tau_{<p-\ell}\Omega^\bullet(E,\theta)\to\mathcal Hom(B^{\mathbb Z_{(p)}}_{\mathrm{HdR},f},F_*\Omega^\bullet(H,\nabla)).
 $$
 It can be reformulated as a morphism of sheaves of sets on $X'$
 $$
 \mathrm{Frame}\times\bigsqcup_{r\geq0}\mathcal L^{r+1}\times B^{\mathbb Z_{(p)}}_{\mathrm{HdR},f}\to \mathcal Hom_{\mathcal O_{X'}}(\tau_{<p-\ell}\Omega^\bullet(E,\theta),F_*\Omega^\bullet(H,\nabla))
$$
which is again denoted by $\mathrm{ev}$.
\end{definition}
\begin{lemma}
 The evaluation morphism above induces a morphism
$$
\mathrm{ev}(-,-,\varphi_p(r,s)):\mathrm{Frame}\times\mathcal L^{r+1}\to\mathcal Hom_{\mathcal O_{X'}}(\tau_{<p-\ell}\Omega^{r+s}(E,\theta),\tau_{p-\ell}F_*\Omega^s(H,\nabla))
$$
which is independent of the first factor $\mathrm{Frame}$. Consequently, we obtain a morphism
$$
\mathrm{ev}(-,\varphi_p(r,s)):\mathcal L^{r+1}\to\mathcal Hom_{\mathcal O_{X'}}(\tau_{<p-\ell}\Omega^{r+s}(E,\theta),\tau_{p-\ell}F_*\Omega^s(H,\nabla)).
$$
 \end{lemma} \begin{proof}
 Given any $r,s\geq 0$ and any data below Definition \ref{scalarization}. By direct computation, we have
 $$
 \begin{array}{rcl}
 \varphi_\infty(r,s)&=&\sum_{(\underline{\underline j},\underline s,\underline i)\in T(r,s)}[\sum_{\overline i\in T(r)}a(\underline{\underline j},\underline s)\underline{\underline j}_{\overline i}\theta^{-1}_{\overline i}\theta^{\underline{\underline j}}e_{\underline i,\overline i}]h^{[\underline{\underline j}]}\zeta_{\underline s,\underline i}\\
 &=&\sum_{(\underline{\underline j},\underline s,\underline i)\in T(r,s)}\sum_{\overline i\in T(r)}a(\underline{\underline j},\underline s)\theta^{\underline{\underline j}-\sum_k\underline{\underline j}(k,i^k)}h^{[\underline{\underline j}-\sum_k\underline{\underline j}(k,i^k)]}e_{\underline i,\overline i}\zeta_{\underline s,\underline i}h_{\overline i}\\
&=&\sum_{\underline{\underline j}\in\underline T(r)}\sum_{\underline s\in P(r,s)\underline i\in T(s),\overline i\in T(r)}(\prod_ls_l!)^{-1}a(\underline{\underline j},\underline s)\theta^{\underline{\underline j}-\sum_k\underline{\underline j}(k,i^k)}h^{[\underline{\underline j}-\sum_k\underline{\underline j}(k,i^k)]}e_{\underline i,\overline i}\zeta_{\underline s,\underline i}h_{\overline i}\\
&=&\sum_{\underline{\underline j}\in\underline N(r)}\sum_{\underline s\in P(r,s), \underline i\in T(s),\overline i\in T(r)}(\prod_ls_l!)^{-1}a(\underline{\underline j}+\sum_k\underline{\underline j}(k,i^k),\underline s)\theta^{\underline{\underline j}}h^{[\underline{\underline j}]}e_{\underline i,\overline i}\zeta_{\underline s,\underline i}h_{\overline i}\\
&=&\sum_{\underline s\in P(r,s)}(\sum_{\underline j\in N(r)}a'(\underline j,\underline s)\prod_k(\sum_l\theta_lh_{k,l})^{[j_k]})(\sum_{ \underline i\in T(s),\overline i\in T(r)}e_{\underline i,\overline i}\zeta_{\underline s,\underline i}h_{\overline i}),
\end{array}
$$
where
\begin{itemize}
\item $h_{\overline{i}}:=\prod_kh_{k,i^k}$, where $h_{k,0}:=1$;
\item $\underline T(r)$ consisting of $\{j_{k,l}\}$ such that $j_k(=\sum_lj_{k,l})>0$ for $k\leq r$ and $=0$ for $k>r$;
\item $\underline N(r)$ consisting of $\{j_{k,l}\}$ such that $j_{k,l}=0$ for $k>r$;
\item $N(r)$ consisting of $\{j_k\}_{k\in\mathbb Z_{>0}}$  such that $j_k=0$ for $k>r$;
\item $P(r,s)$ consisting of $\{s_l\}_{l\in\mathbb Z_{>0}}$ such that $s_l=0$ for $l>r$ and $\sum_ls_l=s$;
\item $a'(\underline j,\underline s)\in\mathbb Z_{(p)}$ which depends on $\underline j\in N(r),\underline s\in P(r,s)$.
\end{itemize}

Clearly, we have
$$
\varphi_p(r,s)=\sum_{\underline s\in P(r,s),\underline j\in N(r),s+j<p}a'(\underline j,\underline s)\prod_k(\sum_l\theta_lh_{k,l})^{[j_k]}(\sum_{ \underline i\in T(s),\overline i\in T(r)}e_{\underline i,\overline i}\zeta_{\underline s,\underline i}h_{\overline i}),
$$
where $s:=\sum_ks_k,j:=\sum_kj_k$. One can check that for any $\underline s\in P(r,s),\underline j\in N(r),s+j<p$, 
$$\mathrm{ev}((\tilde F_0,\cdots,\tilde F_r),\prod_k(\sum_l\theta_lh_{k,l})^{[j_k]}(\sum_{ \underline i\in T(s),\overline i\in T(r)}e_{\underline i,\overline i}\zeta_{\underline s,\underline i}h_{\overline i}))$$ is the restriction of the composite of the following morphisms
$$
\begin{array}{c}
\Omega^{r+s}(E,\theta)_{U'}\stackrel{\rho_1}{\rightarrow}E_{U'}\otimes(\Omega^1_{U'_{\log}/k})^{\otimes(r+s)}\stackrel{\rho_2}{\rightarrow}E_{U'}\otimes F_*\Omega^s_{U_{\log}/k}\stackrel{\rho_3}{\rightarrow}E_{U'}\otimes F_*\Omega^s_{U_{\log}/k}\\
=F_*\Omega^s(H_{\tilde F_0},\nabla_{\tilde F_0})\cong F_*\Omega^s(H,\nabla)_U.
\end{array}
$$
Here
\begin{itemize}
\item $\rho_1$ sends $e\otimes\beta_1\wedge\cdots\wedge\beta_{r+s}\in\Omega^{r+s}(E,\theta)_{U'}$ to $\sum_{\sigma\in\mathcal S_{r+s}} \mathrm{sgn}(\sigma)e\otimes\beta_{\sigma(1)}\otimes\cdots\beta_{\sigma(r+s)}$, where $\mathcal S_{r+s}$ is the group of permutations of $\{1,\cdots,r+s\}$;
\item assume $r,s>0$ and write $\{l|s_l>0\}=\{l_1,\cdots,l_q\}$ with $l_1<\cdots<l_q$. In this case, $\rho_2$ sends $e\otimes\beta_1\otimes\cdots\otimes\beta_{r+s}$ to
$$
e\otimes\zeta_{\tilde F_{k_1}}(\beta_1)\wedge\cdots\wedge\zeta_{\tilde F_{k_s}}(\beta_s)h_{\tilde F_0\tilde F_1}(\beta_{s+1})\cdots h_{\tilde F_{r-1}\tilde F_r}(\beta_{s+r}),
$$
where the sequence $k_1,\cdots,k_s$ is $s_{l_1}$-times $l_1$,$\cdots$, $s_{l_q}$-times $l_q$. The other cases of $r,s$ can be discussed similarly.
\item using the isomorphism $E_{U'}\otimes F_*\Omega^s_{U_{\log}/k}\cong(E_{U'}\otimes F_*\mathcal O_U)\otimes_{F_*\mathcal O_U} F_*\Omega^s_{U_{\log}/k}$, $\rho_3$ is the automorphism given by
$$
\prod_{k=1}^r((\mathrm{id}_{E_{U'}}\otimes h_{\tilde F_{k-1}\tilde F_k})F^*\theta)^{[j_k]}\otimes\mathrm{id}_{F_*\Omega^s_{U_{\log}/k}}.
$$
Here $(\mathrm{id}_{E_{U'}}\otimes h_{\tilde F_{k-1}\tilde F_k})F^*\theta$ is an automorphism of $F^*E_{U'}$ and hence can be regarded as an automorphism of $E_{U'}\otimes F_*\mathcal O_U$.
\end{itemize}
This completes the proof of this lemma.
 \end{proof}

\end{document}